\documentclass[11pt]{amsart}
\usepackage{mypackage}
\usepackage{myformat}
\usepackage{mathtools}

\newcommand{\fseq}[2]{\langle #1_1,\ldots,#1_{#2}\rangle}
\newcommand{\eqfs}[2]{[#1_1,\ldots,#1_{#2}]}
\newcommand{\bw}[3]{ \bigwedge_{#1\in #2}#3_{#1} }

\newcommand{\Hom}[2]{\mathrm{Hom}(#1,#2)}
\begin{document} 
\title[Semilattices]
{Semilattices, Canonical Embeddings and Representing Measures}
\mybic 
\date\today 
\subjclass[2000]{Primary: 28A05, 60G05.} 
\keywords{Lattice, Modular set functions, 
Extension of measures, Non additive integral.}

\begin{abstract} 
We provide conditions under which a modular 
function defined on a semilattice $X$ and with
values in a commutative group is homomorphic to
a modular function on a lattice $L$ for any 
embedding $X\hookrightarrow L$.
\end{abstract}

\maketitle

\ToWhom{
This paper is dedicated to the memory of 
Filippo Arpesani.
}

\section{Introduction and Notation} 
In this paper we study functions  that map a 
semilattice, i.e. an idempotent, commutative 
semigroup, into a commutative group. Similar
objects abund in measure theory where the
main example is that of a vector valued set 
function defined on a family of sets with 
minimal properties such as the family of 
convex sets in a real vector space or the 
family of regular open sets in a topological 
space. In analysis a general example of a
semilattice is the family of real valued 
functions defined on some topological space
and with compact support. The main questions 
we address in this paper are 
\tiref i 
whether a semilattice may be embedded into a 
proper lattice and 
\tiref{ii} 
under what conditions a function defined on a
semilattice and with values in a group admits 
an extension to such larger domain that preserves 
reasonable properties. In particular our interest 
is for modular functions a property near to 
additivity arising naturally in the study of
capacities initiated by Choquet \cite{choquet}.

Our results are based on the construction of an
embedding of a semilattice into a special space
of sequences on which a lattice structure is easily
defined. This embedding is canonical in the sense 
that every other embedding into a lattice is isomorphic
to it. Moreover every function defined on a semilattice
may be homomorphically represented as a function
on its canonical space. Based on this construction
we are able to associate to each function a 
 corresponding set function, its {\it distribution},
and exploit some measure extension techniques
to eventually obtain a positive answer to question 
\tiref{ii}.

The extension of an additive function defined on some 
family of sets to a larger domain is a time honoured problem 
that has received many important contributions,
among which one should mention Horn and Tarski 
\cite{horn_tarski} and Pettis \cite{pettis}. However, 
their results are specific to set functions and, at least
in the case of Horn and Tarski, exploit outer measure
techniques which are simply not applicable when
the function considered takes value in a group 
without an order structure.

On the other hand, functions having a general 
semilattice as their domain have appeared in a 
limited number of noteworthy papers by authors
such as Leader \cite{leader}, Newman \cite{newman} 
and Kist and Maserick \cite{kist_maserick}. In these
papers the focus is on real valued functions and
the aim is to extend to this context the familiar
notions of monotonicity and of variation. 

After some combinatorial preliminaries in section 
\ref{sec Mobius}, we introduce in section \ref{sec modular} 
modular set functions and obtain the canonical 
embedding of a semilattice in section \ref{sec canonical}. 
Eventually, in section \ref{sec extension} we prove our 
result on the extension of a modular function defined on 
a semilattice.

We find convenient the following general notation. 
When $S$ is a set, $\abs S$ is its cardinality and 
mapping a set $X$ into another set $Y$ is
written as $\Fun{X,Y}$ (or simply $\Fun X$ when 
$Y=\R$) and the image of $X$ in $Y$ under some 
$f\in\Fun{X,Y  }$ is denoted by $f[X]$. If $N\in\N$ we 
write $[N]=\{1,\ldots,N\}$. 

\section{M\"obius preliminaries.}
\label{sec Mobius}

Let $\mathcal X$ be the collection of all finite subsets 
of a given set $X$, ordered by inclusion and write
\begin{equation}
\label{nu}
\nu(b)
	=
(-1)^{1+\abs b}
\qquad
b\in\mathcal X.
\end{equation}
If $j\in\N$, we write $\nu(j)=\nu([j])$. If $a\le b$, 
one easily concludes that $\nu(b)=\mu(a,b)\nu(a)$ 
where $\mu$ is the M\"{o}bius function on $\mathcal X$, 
see \cite[Corollary, p. 345]{rota}. We thus deduce 
from \cite[Propositions 1 and 2, p. 344]{rota} 
that $\nu$ is the unique set function on $\mathcal X$ 
satisfying: 
\begin{subequations}
\begin{equation}
\label{Mobius function}
\nu(\emp)
	=
-1
\qtext{and}
\sum_{a\le t\le b}\nu(t)
	=
0
\qquad a,b\in\mathcal X,\ a\ne b
\end{equation}
\begin{equation}
\label{Mobius inversion}
f(b)
	=
\nu(b)\sum_{\{t:a\le t\le b\}}F(a,t)\nu(t)
\qtext{and}
f(a)
	=
\nu(a)\sum_{\{s:a\le s\le b\}}F(s,b)\nu(s)
\qquad
f:\mathcal X\to\R
\end{equation}
\end{subequations}
where $F(a,b)=\sum_{a\le y\le b}f(y)$. 

A useful application of the M\"{o}bius function
is the following Lemma that will be very useful
in the sequel.

\begin{lemma}
\label{lemma combinatorial} 
Let $A$ and $B$ be finite sets and denote by $K_j(A,B)$ 
the number of subsets of $A\times B$ with cardinality 
$j$ and in which each $a\in A$ and $b\in B$ appear 
at least once. Then,
\begin{equation} 
\label{K}
\sum_{0<j\le\abs{A\times B}}\nu(j)K_j(A,B)
	=
\nu(A)\nu(B).
\end{equation}
\end{lemma}

\begin{proof} 
Forming subsets of $A\times B$ of cardinality 
$j>0$ by using either \tiref{i} {\it at most} or 
\tiref{ii} {\it exactly} $n$ elements of $A$ and 
$m$ elements of $B$ may be achieved in a 
number of ways, (\textit{i}) $q_j(n,m)$ or 
(\textit{ii}) $k_j(n,m)$ respectively. Of 
course, $q_j(n,m)=k_j(n,m)=0$ unless 
$j\le nm$. Moreover, by classical 
combinatorial formulas we obtain
\begin{align}
\label{q}
\sum_{\substack{0<t\le n\\0<s\le m}}k_j(t,s)
	=
q_j(n,m)
	=
{nm\choose j}{\abs A\choose n}{\abs B\choose m}
\end{align}
which can be inverted, via \eqref{Mobius inversion}, 
to give%
\footnote{
We set as usual ${n\choose k}=0$ if $k>n$.
}
\begin{align*}
k_j(\abs A,\abs B)
	&=
\nu(A)\nu(B)
\sum_{\substack{
0<n\le\abs A
\\
0<m\le\abs B}}
\nu(n)\nu(m)q_j(n,m)
\\
	&=
\nu(A)\nu(B)\sum_{\substack{
0<n\le\abs A
\\
0<m\le\abs B}}
(-1)^{n+m}{nm\choose j}{\abs A\choose n}{\abs B\choose m}.
\end{align*}
Since 
$K_j(A,B)=k_j(\abs A,\abs B)$ 
and 
$\sum_{0<j\le N}\nu(j){N\choose j}=1$,
we conclude
\begin{align*}
\sum_{0<j\le\abs{A\times B}}\nu(j)K_j(A,B)
	&=
\nu(A)\nu(B)\sum_{\substack{0<n\le\abs A\\0<m\le\abs B}}
(-1)^{n+m}{\abs A\choose n}{\abs B\choose m}
\sum_{0<j\le nm}\nu(j){nm\choose j}
	=
\nu(A)\nu(B).
\end{align*}
\end{proof}

\section{ Semilattices and modular functions}
\label{sec modular}

Following Birkhoff \cite{birkhoff}, we define a 
semilattice as a commutative, idempotent 
semigroup $(X,\wedge)$ endowed with the 
partial order defined by letting $x\le y$ whenever 
$x=x\wedge y$. Then, $x\wedge y$ coincides 
with the greatest lower bound of the set 
$\{x,y\}$, so that $X$ is a semilattice in the
order theoretic sense. Conversely, a semilattice 
with respect to some order relation is a 
commutative, idempotent semigroup 
relatively to the semilattice operation. 
Thus a semilattice may be equivalently 
defined as a partially ordered set in which 
every pair admits a greatest lower bound.

\begin{example}
If for each $\alpha$ in a non empty index set
$\mathfrak A$ the pair $(X_\alpha,\wedge_\alpha)$
is a semilattice, then the product semilattice
$\bigtimes_{\alpha\in\mathfrak A}(X_\alpha,\wedge_\alpha)$
is defined as 
$(\bigtimes_{\alpha\in\mathfrak A}X_\alpha,\wedge)$
where $\wedge$ is defined by letting
\begin{equation}
(x_\alpha)_{\alpha\in\mathfrak A}
\wedge
(x'_\alpha)_{\alpha\in\mathfrak A}
	=
(x_\alpha \wedge_\alpha x'_\alpha)_{\alpha\in\mathfrak A}.
\end{equation}
A special case of this construction is
$(X,\cap)\times(X',\cup)$ with $X$ a $\cap$ closed
family of subsets of $\Omega$ and $X'$ a
$\cup$ closed family of subsets of $\Omega'$.
\end{example}

A mapping $\iota:X\to X'$ between two semilattices 
is a (semilattice) homomorphism -- in symbols
$\iota\in\Hom X{X'}$ -- if 
$\iota(x\wedge y)
	=
\iota(x)\wedge\iota(y)$ 
for all $x,y\in X$. 
A semilattice homomorphism which is injective 
is referred to as a semilattice embedding and it
establishes an isomorphism onto its range.
Homomorphisms preserve order. If $X$ is a 
semilattice, the least upper bound of a set 
$a=\{x_1,\ldots,x_N\}$ may occasionally exist 
in $X$, in which case we denote it by 
$\bigvee_{n=1}^Nx_n$ or by $\vee a$. It is easily seen that if
$\iota\in\Hom{X}{X'}$ is a semilattice {\it isomorphism}, 
then $\bigvee_{n=1}^Nx_n$ exists in $X$ if and 
only if $\bigvee_{n=1}^N\iota(x_n)$ exists in 
$X'$ and 
$\bigvee_{n=1}^N\iota(x_n)
	=
\iota\big(\bigvee_{n=1}^Nx_n\big)$.
However, if $\iota$ is just an embedding the 
existence of $\bigvee_{n=1}^Nx_n$ in $X$ does 
not guarantee that of $\bigvee_{n=1}^N\iota(x_n)$ 
in $X'$ and, in case both exist, one can only
establish the inequality 
$
\iota\big(\bigvee_{n=1}^Nx_n\big)
	\ge
\bigvee_{n=1}^N\iota(x_n)
$%
\footnote{
In the special case in which $X$ is a $\wedge$ 
closed subset of another semilattice $X'$, the 
supremum of the set $\{x,y\}\subset X$ 
may be computed relatively to $X$ or to $X'$. 
To avoid ambiguity we shall always write the
former as $x\vee y$ and the latter as
$\iota(x)\vee\iota(y)$, with $\iota$ the
inclusion map $\iota:X\to X'$.
}.

A distinguished property of functions defined on
semilattices is modularity relatively to some Abelian
group $G$ written additively, as customary.

\begin{definition}
Let $X$ be a semilattice and $G$ an Abelian group. 
Then $f\in\Fun{X,G}$ is said to be modular -- in symbols
$f\in\Fun[m]{X,G}$ -- whenever
\begin{align}
\label{modular}
f\Big(\bigvee_{1\le n\le N}x_n\Big)
	&=
\sum_{\emp\ne b\le[N]}\nu(b)f\Big(\bw n b x\Big)
\qquad
x_1,\ldots,x_N,\bigvee_{1\le n\le N}x_n\in X.
\end{align}
\end{definition}

A useful implication of \eqref{modular} is that if 
$f$ is a modular function on a lattice $L$ and if 
$L$ is generated by some $\wedge$ closed subset 
$X\subset L$, namely if
$L
	=
\{\bigvee_{1\le n\le N}x_n:x_1,\ldots,x_N\in X\}$,
then $f$ is completely determined by its restriction
to $X$. In other words, if $X$ generates $L$ then
each $g\in\Fun{X,G}$ admits at most one modular 
extension to $L$.

The following claim is obvious given the preceding
remarks.

\begin{lemma}
\label{lemma comp}
Let $\iota\in\Hom XY$ be a semilattice isomorphism,
$G$ an Abelian group and $f\in\Fun{Y,G}$. Then
$f$ is modular if and only if so is $f\circ\iota$.
\end{lemma}

In case $X$ is a lattice, property \eqref{modular} 
takes the more familiar form
\begin{equation}
\label{2 additive}
f(x\vee y)
	=
f(x)+f(y)-f(x\wedge y)
\qquad
x,y\in X.
\end{equation}
A function satisfying \eqref{2 additive} is called a
valuation by some authors while others refer to 
\eqref{2 additive} as two-additivity. It is easy to
show that \eqref{2 additive} is in turn equivalent 
to%
\footnote{
The proof of the equivalence of \eqref{2 additive} 
with \eqref{additive} was first obtained by Horn and 
Tarski \cite[Corollary 1.3]{horn_tarski}, although in an 
apparently different context, and later rediscovered 
by Leader \cite[Proposition 10]{leader}.
}
\begin{equation}
\label{additive}
\sum_{1\le n\le N}f(x_n)
	=
\sum_{1\le k\le N}
f\Big(\bigvee_{\{b\le[N]:\abs b=k\}}\bw i b x\Big)
\qquad
x_1,\ldots,x_N\in X.
\end{equation}

The simple condition \eqref{2 additive} is not enough
to guarantee that a function is modular when $X$ is 
just a semilattice since in this setting we 
may encounter the somewhat pathological situation
in which $\bigvee_{n=1}^Nx_n\in X$ 
while $\bigvee_{n=1}^{N-1}x_n\notin X$. 
Nevertheless, \eqref{2 additive} is (necessary and) 
sufficient for modularity in some special cases of 
interest.

\begin{lemma}
Let $X$ be a semilattice and $G$ an Abelian group. 
If $X$ is such that
\begin{equation}
\label{groemer}
\bigvee_{1\le n\le N}(x_n\wedge x)\in X
\qquad
x,x_1,\ldots,x_N\in X
\end{equation}
then \eqref{modular} is equivalent to \eqref{2 additive} 
for all $x,y\in X$ with $x\vee y\in X$.
\end{lemma}

\begin{proof}
Let $N$ be the smallest integer for which there 
exists a finite family $x_1,\ldots,x_N\in X$ such that
\begin{equation}
\bigvee_{1\le n\le N}x_n\in X
\qtext{but}
f\Big(\bigvee_{1\le n\le N}x_n\Big)
\ne
\sum_{\emp\ne b\le[N]}\nu(b)f\Big(\bw n b x\Big).
\end{equation}
By \eqref{groemer}, $z=\bigvee_{1\le n\le N-1}x_n$ 
exists in $X$ so that, under \eqref{2 additive},
\begin{align*}
f\Big(\bigvee_{1\le n\le N}x_n\Big)
&=
f(z)+f(x_N)-f(z\wedge x_N)
\\&=
\sum_{\emp\ne b\le[N-1]}\nu(b)f\Big(\bw n b x\Big)
+
f(x_N)
-
\sum_{\emp\ne a\le[N-1]}\nu(a)f\Big(\bw n a x\wedge x_N\Big)
\\&=
\sum_{\emp\ne b\le[N]}\nu(b)f\Big(\bw n b x\Big),
\end{align*}
a contradiction.
\end{proof}

If $(Y,\cdot)$ is a commutative semigroup, the so is 
$\Fun{X,Y}$ upon defining
\begin{equation}
(f\cdot g)(x)
	=
f(x)\cdot g(x)
\qquad
x\in X,\ 
f,g\in\Fun{X,Y}.
\end{equation}
It makes thus sense to claim that $\Fun{X,G}$
and $\Fun{X',G}$ are (semigroup) homomorphic.

%We close this section with a result on factorization.
%If $X$ is a semilattice, $G$ an Abelian group and 
%$f\in\Fun{X,G}$ then we define the equivalence
%relation $\sim_f$ on $X$ by letting
%\begin{equation}
%\label{equivalence}
%x\sim_fx'
%\qtext{whenever}
%f(x\wedge z)
%	=
%f(x'\wedge z)
%\qquad
%z\in X.
%\end{equation}
%Clearly $x\sim_fx'$ implies $f(x)=f(x')$. It is then
%natural to define the set $X/f$ of all equivalence
%classes $x/f$ of $X$ relative to $\sim_f$. Given
%that $z\sim_fx$ and $z'\sim_fx'$ imply
%$(z\wedge z')\sim_f(x\wedge x')$ we may define
%the binary operation $\wedge$ on $X/f$ by
%letting $(x/f)\wedge(x'/f)=(x\wedge x')/f$. If 
%$X$ were a lattice the operation $\vee$ on $X/f$
%may be defined likewise. It is easily seen that with 
%these definitions in place $(X/f,\wedge)$ is 
%itself a semilattice. More interestingly:
%
%\begin{lemma}
%\label{lemma factor}
%Let $X$ be a lattice and $f\in\Fun{X,G}$. There 
%exists $\tilde f\in\Fun{X/f,G}$ such that
%\begin{equation}
%\label{quot}
%\tilde f\big(x/f\big)
%	=
%f(x)
%\qquad
%x\in X.
%\end{equation}
%If $X$ is a lattice, then $f$ is modular if and only
%if so is $\tilde f$.
%\end{lemma}

%\begin{proof}
%It is clear that \eqref{quot} implicitly provides a
%definition of $\tilde f\in\Fun{X/f,G}$. Suppose that
%$\tilde f$ is modular and let $x_1,\ldots,x_N\in X$
%be such that $\bigvee_{1\le n\le N}x_n\in X$.
%\end{proof}

\section{The canonical embedding of a semilattice}
\label{sec canonical}

Throughout the rest of the paper $X$ will be a fixed 
semilattice and $G$ a given Abelian group.

For two finite sequences of elements in $X$ we write 
\begin{equation}
\label{seq order}
\fseq xN
	\le
\fseq zM
\qtext{whenever}
\big(\forall n\in[N]\big)\big(\exists m\in[M]\big):
x_n\le z_m.
\end{equation}
Of course, if $\emp$ denotes the empty sequence,
then trivially $\emp\le\fseq xN$ for all finite 
sequence $\fseq xN$ in $X$. The equivalence 
relation implicit in this definition induces the 
quotient space $\Xi$ of all finite sequences from 
$X$. The obvious extension of $\ge$ to $\Xi$ defines 
a partial order. The symbol $\eqfs xN$ denotes the 
equivalence class containing the finite sequence
$\fseq x N$. A special representative of the class 
$\xi\in\Xi$, indicated by $\xi^*$, is the one having
a minimal number of elements. Eventually if 
$\xi=\eqfs xN$ and $\zeta=\eqfs zM$, then we define
the binary operations
\begin{equation}
\label{Xi lattice}
\xi\wedge\zeta
	=
[x_1\wedge z_1,\ldots,x_n\wedge z_m,\ldots, x_N\wedge z_M]
\qtext{and}
\xi\vee\zeta
	=
[x_1,\ldots,x_N,z_1,\ldots,z_M]
\end{equation}
which makes $\Xi$ into a (distributive) lattice.

The following, simple result illustrates the importance
of this construction. 
%We find it convenient, when
%$x_1,\ldots,x_N\in X$ and $b\subset\N$, to write
%$x(b)=\{x_n:n\in b\cap[N]\}$.

\begin{theorem}
\label{th canonical}
$X$ embeds into the lattice $\Xi$ of equivalence 
classes of finite sequences from $X$. $\Fun{X,G}$ 
is group isomorphic to
$\{T\in\Fun[m]{\Xi,G}:T(\emp)=0\}$ 
via the identity
\begin{equation}
\label{canonical}
T(\xi)
	=
\sum_{\emp\ne b\le[N]}\nu(b)
f\Big(\bigwedge_{n\in b} x_n\Big)
\qquad
\xi=\eqfs xN\in\Xi.
\end{equation}
Moreover, 
(i)
$f$ is modular if and only if 
$T(\eqfs xN)=T([\bigvee_{1\le n\le N}x_n])$
whenever $\bigvee_{1\le n\le N}x_n\in X$,
(ii)
if $G$ is partially ordered then $T$ is monotonic 
if and only if
\begin{equation}
\label{monotonic}
T(\xi)
	\le
f(x)
\qtext{whenever}
\xi\le[x].
\end{equation}
\end{theorem}

\begin{proof}
The mapping $x\to[x]$ of $X$ into 
$\Xi_0
	=
\{[x]:x\in X\}
\subset
\Xi$ 
is clearly a semilattice isomorphism and induces a 
bijection between $\Fun{X,G}$ and $\Fun{\Xi_0,G}$ 
in an obvious way. To prove that the identity 
\eqref{canonical} 
establishes an isomorphism, assume that 
$T\in\Fun[m]{\Xi,G}$ and that $T(\emp)=0$ 
and define $f\in\Fun{X,G}$ via
\begin{equation}
\label{Ttof}
f(x)
	=
T([x])
\qquad
x\in X.
\end{equation}
If $\xi=\eqfs x N$, then $\xi=\bigvee_{n=1}^N[x_n]$
and, since $T$ is modular,
\begin{align*}
T(\xi)
	=
\sum_{\emp\ne b\le[N]}\nu(b)
T\Big(\bigwedge_{n\in b}[x_n]\Big)
	=
\sum_{\emp\ne b\le[N]}\nu(b)
T\Big(\big[\bw nbx\big]\Big)
	=
\sum_{\emp\ne b\le[N]}\nu(b)
f\Big(\bw nbx\Big).
\end{align*}
This shows that $f$ as defined in \eqref{Ttof}
satisfies \eqref{canonical}; moreover, 
since $\Xi_0$ generates $\Xi$, the map 
$T\to f$ is injective. To show that it is onto,
let $f\in\Fun{X,G}$, fix $x_1,\ldots,x_N\in X$
and designate the right hand side of \eqref{canonical} 
by $F(\fseq xN)$. In order to show that
$F(\fseq xN)=F(\fseq zM)$ when
$\eqfs xN=\eqfs zM$ we start 
noting that, since $G$ is Abelian, $F$ is 
invariant with respect to a permutation of 
the elements of the sequence. Second,
suppose that $x_N\le x_{N-1}$. Then,
\begin{align*}
F(\fseq xN)
&= 
F(\fseq x{N-1})
+
f(x_N)
-
F(\fseq{x_N\wedge x}{N-1})
\\&=
F(\fseq x{N-1})
+
f(x_N)
-
F(\fseq{x_N\wedge x}{N-2})
\\&\quad-f(x_{N-1}\wedge x_N)
+
F(\fseq{x_N\wedge x_{N-1}\wedge x}{N-2})
\\&=
F(\fseq x{N-1}).
\end{align*}
It follows that the value of $F(\fseq xN)$ is
unaffected if we drop the dominated elements
from the sequence $\fseq xN$. Therefore,
\begin{equation}
F(\fseq xN)
	=
F(\xi^*),
\qquad
\xi=\eqfs xN\in\Xi.
\end{equation}
This shows that indeed \eqref{canonical} defines
an element $T\in\Fun{\Xi,G}$ such that
$T(\emp)=\sum\emp=0$. To prove that
$T$ is modular let $\xi=\eqfs xN$, $\zeta=\eqfs zM$ 
and write $\xi\vee\zeta$ as $\eqfs y K$.
For any $\alpha\le[N]$ and $\beta\le[M]$ 
denote by $\Gamma_j(\alpha,\beta)$ the family 
of subsets of $\alpha\times\beta$ consisting of 
exactly $j$ pairs $(n,m)$ and in which all 
$n\in\alpha$ and $m\in\beta$ appear at least 
once. The cardinality of $\Gamma_j(\alpha,\beta)$ 
was denoted by $K_j(\alpha,\beta)$ in Lemma 
\ref{lemma combinatorial}. Observe that 
$\Gamma_j(\alpha,\beta)\ne\emp$ if and
only if 
$\abs\alpha\vee\abs\beta
\le
j
\le
\abs\alpha\abs\beta$,
that
\begin{equation}
\label{cross}
\bigwedge_{(n,m)\in\gamma}(x_n\wedge z_m)
	=
\bw n\alpha x\wedge\bw m\beta z
\qquad
\gamma\in\Gamma_j(\alpha,\beta),
\abs\alpha\vee\abs\beta\le j\le\abs\alpha\abs\beta
\end{equation} 
and that 
$\nu(\abs\alpha\abs\beta)
=
-\nu(\alpha)\nu(\beta)$.

\begingroup
\allowdisplaybreaks
\begin{equation}
\begin{split}
\label{long}
T(\xi\vee\zeta)
	&=
\sum_{\substack{\alpha\le[N],\ \beta\le[M]\\\emp\ne\alpha\cup\beta}}
\nu(\abs\alpha+\abs\beta)
f\big(\bw n\alpha x\wedge\bw m\beta z\big)
\\
	&=
T(\xi)+T(\zeta)
-
\sum_{\substack{\emp\ne\alpha\le[N]\\ \emp\ne\beta\le[M]}}
\nu(\alpha)\nu(\beta)
f\big(\bw n\alpha x\wedge\bw m\beta z\big)
\\
	&=
T(\xi)+T(\zeta)
-
\sum_{\substack{\emp\ne\alpha\le[N]\\ \emp\ne\beta\le[M]}}\ 
\sum_{0<j\le\abs\alpha\abs\beta}\nu(j)K_j(\alpha,\beta)
f\big(\bigwedge_{\substack{n\in\alpha\\m\in\beta}}(x_n\wedge z_m)\big)
\qquad(\text{by }\eqref{K})
\\
	&=
T(\xi)+T(\zeta)
-
\sum_{\substack{\emp\ne\alpha\le[N]\\ \emp\ne\beta\le[M]}}\ 
\sum_{0<j\le\abs\alpha\abs\beta}
\sum_{\gamma\in\Gamma_j(\alpha,\beta)}
\nu(\gamma)f\big(\bw k\gamma y\big)
\hspace{2.2cm}(\text{by }\eqref{cross})
\\
	&=
T(\xi)
+
T(\zeta)
-
\sum_{\emp\ne\gamma\le[K]}
\nu(\gamma)f\big( \bw k\gamma y\big)
\\
	&=
T(\xi)+T(\zeta)-T(\xi\wedge\zeta).
\end{split}
\end{equation}
\endgroup
Thus, $T\in\Fun[m]{\Xi,G}$ so that the map $f\to T$
defines a bijection between the groups $\Fun{X,G}$ 
and $\{T\in\Fun[m]{\Xi,G}:T(\emp)=0\}$ which is 
clearly additive and therefore a group 
isomorphism.

By definition $f$ is modular if and only if for 
all $x_1,\ldots,x_N\in X$ with 
$\bigvee_{1\le n\le N}x_n\in X$ one has
$T([\bigvee_{1\le n\le N}x_n])
=
f(\bigvee_{1\le n\le N}x_n)
=
T(\eqfs xN)$.
If $T$ is monotonic then \eqref{monotonic} is obvious. 
Conversely, let \eqref{monotonic} hold and choose
$\xi=\eqfs xN\le\eqfs z K=\zeta$. We can then assume, with
no loss of generality that $x_1\le z_1$. Then
by modularity
\begin{equation}
\label{Tmon}
T([z_1,x_2,\ldots,x_N])
	=
T([z_1,x_1\ldots,x_N])
	=
T(\xi)+f(z_1)-T([x_1\wedge z_1,\ldots,x_N\wedge z_1])
\ge
T(\xi).
\end{equation}
Replacing iteratively, each $x_n$ with some $z_k$ 
dominating it and then adding the remaining terms
of $\zeta$ we progressively increase the value
of the left hand side of \eqref{Tmon} until we 
reach $T(\zeta)$.
\end{proof}

\begin{remark}
Notice that \eqref{canonical} may be inverted via
\eqref{Mobius inversion} to give
\begin{equation}
f\Big(\bigwedge_{n\in b}x_n\Big)
	=
\sum_{\emp<a\le b}\nu(a)T\Big(\bigvee_{j\in a}[x_j]\Big)
\qquad
x_1,\ldots,x_N\in X.
\end{equation}
\end{remark}

\begin{remark}
For the case $G=\R$ one encounters in the literature
several different conditions that turn out to be equivalent 
to \eqref{monotonic}. Leader \cite[Proposition 6, p. 413]
{leader} focuses on $f$ being positive 
definite, i.e.
\begin{equation}
\label{pos def}
\sum_{i,j}a_if(x_i\wedge x_j)a_j
	\ge
0
\qquad
a_1,\ldots,a_N\in\mathbb Z,\ 
x_1,\ldots,a_N\in X.
\end{equation}
In his fundamental study of capacities Choquet 
\cite[p. 149]{choquet} introduced the difference 
operators $\Delta_0,\Delta_1,\ldots$ by letting 
$\Delta_0f(a)=f(a)$ and%
\footnote{
In Choquet setting $f$ is a real valued set function
on a $\cup$ closed family of sets.
}
\begin{equation}
\Delta_nf(a;x_1,\ldots,x_n)
=
\Delta_{n-1}f(a;x_1,\ldots,x_{n-1})
-
\Delta_{n-1}f(a\wedge x_n;x_1,\ldots,x_{n-1})
\qquad
n\in\N
\end{equation}
and defines a capacity to be monotonic when all of
its differences are positive. This definition is the 
discrete analogue of that of complete monotonicity 
for functions. It is easily established by induction 
that $\Delta_0f(a)=T([a])$ and
\begin{equation}
\Delta_nf(a;x_1,\ldots,x_n)
=
T([a,x_1,\ldots,x_n])-T(\eqfs x n)
=
f(a)-T([a\wedge x_1,\ldots,a\wedge x_n]).
\end{equation}
Therefore, $T$ is monotonic if and only if 
$f$ is completely monotonic. 
\end{remark}

We will refer to the embedding $X\hookrightarrow\Xi$ 
as the {\it canonical} embedding of $X$,
to $\Xi$ as the canonical space of $X$ and to $T$ in 
\eqref{canonical} as the canonical representation of 
$f$. The qualification as canonical derives from the
embedding $X\hookrightarrow\Xi$ being independent
of the choice of $G$. As we have seen, several properties 
of $f$ translate into a corresponding property of $T$. 
In particular, we deduce from \tiref i that $f$ may fail 
to be modular only if the canonical embedding of $X$ 
does not commute with the $\vee$ binary operation, 
whenever well defined. More precisely, 
\begin{lemma}
\label{lemma Norberg}
The following are equivalent:
\begin{enumerate}[(a)]
\item
the canonical embedding of $X$ commutes with $\vee$,
\item
\label{merzbach}
$\bigvee_{n=1}^Nx_n\in X$ if and only if
$x_{n_0}=\bigvee_{n=1}^Nx_n$ for some
$1\le n_0\le N$,
\item 
$\Fun[m]{X,G}=\Fun{X,G}$.
\end{enumerate}
\end{lemma}

\begin{proof}
Assume that $x_1,\ldots,x_N$ and 
$\bigvee_{n=1}^Nx_n$ are elements of $X$.
Under \tiref{a}, 
$[x_1,\ldots,x_N]
=
\bigvee_{n=1}^N[x_n]
=
\big[\bigvee_{n=1}^Nx_n\big]
$.
By definition of the order on $\Xi$, this is equivalent
to $\bigvee_{n=1}^Nx_n
=
x_{n_0}$ for some $1\le n_0\le N$. 
This shows that \tiref a and \tiref b are equivalent 
properties.
If $G$ is an Abelian group, $f\in\Fun{X,G}$ and
$\bigvee_{n=1}^Nx_n\in X$, then \tiref{b} implies
$f\big(\bigvee_{n=1}^Nx_n\big)
=
T([x_1,\ldots,x_N])$ so that $f\in\Fun[m]{X,G}$.
Conversely, let $x_1,\ldots,x_N\in X$ and 
$x=\bigvee_{n=1}^Nx_n\in X$ but $x\ne x_n$ for 
all $1\le n\le N$. Fix $0\ne y\in G$ and define 
$f\in\Fun{X,G}$ implicitly by letting $f(z)=y$ if 
$x\le z$ or else $f(z)=0$. Then, for each 
$\emp\ne b\subset[N]$ we conclude 
$f(\bigwedge_{n\in b}x_n)=0$ while $f(x)=y$ 
so that $f\notin\Fun[m]{X,G}$.
\end{proof}

Condition  \eqref{merzbach} of Lemma \ref{lemma Norberg}
first (tacitly) appeared in an unpublished work of 
Norberg \cite{norberg} and is a crucial assumption in
the so called theory of set-indexed stochastic processes. 
It asserts that there exists no non trivial least upper 
bound and it is though very restrictive in some special 
cases, e.g. when $X$ is a lattice in which case it is 
equivalent to $X$ being linearly ordered.

\begin{example}
Let $\tau$ and $\sigma$ be stopping times on a filtered 
probability space such that $0<P(\sigma\le\tau)<1$. 
Consider the $\cap$ semilattice of stochastic intervals
of the form $[[0,T]]$ as defined e.g. in 
\cite[p. 49]{dellacherie}. Then
$
[[0,\tau]]\cup[[0,\sigma]]
	=
[[0,\tau\vee\sigma]]$
but $P(\tau\vee\sigma>\tau)>0$ 
and 
$P(\tau\vee\sigma>\sigma)>0$ 
so that condition \tiref{ii} of Lemma \ref{lemma Norberg}
fails. This shows that this property fails even for stochastic 
processes indexed by stopping times.
\end{example}

The canonical embedding is the model for any other
embedding of a semilattice into a lattice.

\begin{theorem}
\label{th embed}
If $L$ is a lattice then $\Hom XL$ and $\Hom\Xi L$ 
are lattice isomorphic and the corresponding elements
are related via the equation
\begin{equation}
\label{iso}
\bigvee_{1\le n\le N}h(x_n)
	=
\iota([x_1,\ldots,x_N])
\qquad
x_1,\ldots,x_N\in X.
\end{equation}
\end{theorem}

\begin{proof}
If $\iota\in\Hom\Xi L$ define $h=\iota\circ\kappa$
with $\kappa$ the canonical map of $X$. Then,
\eqref{iso} follows from the fact that $\iota$ is a 
lattice homomorphism so that
\begin{align*}
\bigvee_{1\le n\le N}h(x_n)
=
\bigvee_{1\le n\le N}\iota([x_n])
=
\iota\Big(\bigvee_{1\le n\le N}[x_n]\Big)
=
\iota(\eqfs xN).
\end{align*}
Moreover, given that $\Xi$ is generated by $\kappa[X]$, 
the map $\iota\to h$ of $\Hom\Xi L$ into $\Hom XL$ is 
injective. On the other hand, if $h\in\Hom XL$ and 
$\xi=\eqfs xN$ and $\zeta=\eqfs zK$ are elements 
of $\Xi$ we deduce from the definition of order on 
$\Xi$ that
\begin{equation}
\xi
	\le
\zeta
\qtext{implies}
\bigvee_{1\le n\le N}h(x_n)
	\le
\bigvee_{1\le k\le K}h(z_k).
\end{equation}
We may thus define the map $\iota:\Xi\to L$ 
by letting
\begin{equation}
\label{chi}
\iota(\eqfs x N)
	=
\bigvee_{1\le n\le N}h(x_n)
\qquad
x_1,\ldots,x_N\in X.
\end{equation}
Moreover, if $\xi,\zeta\in\Xi$ are as above, then
\begin{align*}
\iota(\xi\vee\zeta)
	&=
\iota([x_1,\ldots,x_N,z_1,\ldots,z_K])
	=
\bigvee_{\substack{1\le n\le N\\1\le j\le K}}
h(x_n)\vee h(z_j)
	=
\iota(\eqfs xN)\vee\iota(\eqfs z K)
\end{align*}
as well as (using 
$h(x_n\wedge z_j)
	=
h(x_n)\wedge h(z_j)$)
\begin{align*}
\iota(\xi\wedge\zeta)
	&=
\iota([x_1\wedge z_1,\ldots,x_n\wedge z_k,\ldots,x_N\wedge z_K])
	=
\bigvee_{\substack{1\le n\le N\\1\le j\le K}}
h(x_n)\wedge h(z_j)
	=
\iota(\xi)\wedge\iota(\zeta).
\end{align*}
Thus $\iota\in\Hom\Xi L$ and the map $\Hom\Xi L\to\Hom XL$ 
is onto as well. If $h_j$ and $\iota_j$ are in correspondence 
via \eqref{iso} for $j=1,2$, then
\begin{align*}
(h_1\wedge h_2)(x)
=
(\iota_1\circ\kappa)(x)\wedge(\iota_2\circ\kappa)(x)
=
\iota_1(\kappa(x))\wedge\iota_2(\kappa(x))
=
(\iota_1\wedge\iota_2)(\kappa(x))
=
\big((\iota_1\wedge\iota_2)\circ\kappa\big)(x)
\end{align*}
for all $x\in X$, and similarly for $\vee$. This proves
that the bijection between $\Hom XL$ and $\Hom\Xi L$ 
is a lattice isomorphism.
\end{proof}

The next, crucial result justifies interest in the
modular property.

\begin{theorem}
\label{th ext}
Let $X$ be a $\wedge$ closed subset of a lattice 
$L$, $\iota:X\hookrightarrow L$ the inclusion map 
and $L_X$ the sublattice of $L$ generated by
$\iota[X]$. Each $f\in\Fun[m]{X,G}$ admits
a unique extension $g\in\Fun[m]{L_X,G}$. 
Moreover, if $T$ is the canonical representation 
of $f$ then $g$ takes the form
\begin{equation}
\label{ext}
g\Big(\bigvee_{1\le n\le N}\iota(x_n)\Big)
	=
T(\eqfs xN)
\qquad
x_1,\ldots,x_N\in X.
\end{equation}
\end{theorem}

\begin{proof}
Let $T$ be the canonical representation of $f$
and choose $\eqfs x N,\eqfs zK\in\Xi$ such that 
$
\bigvee_n\iota(x_n)
	=
\bigvee_k\iota(z_k)
$.
Then, 
$\bigvee_k\iota(z_k\wedge x_1)$
exists in $X$ and is equal to $x_1$. Given
that $f$ is modular,
\begin{equation}
f(x_1)
	=
T([z_1\wedge x_1,\ldots,z_K\wedge x_1])
	=
T(\eqfs zK)+T([x_1])-T([z_1,\ldots,z_K, x_1])
\end{equation}
and therefore
$T(\eqfs zK)
	=
T([z_1,\ldots,z_K, x_1])$.
Proceeding iteratively we establish
\begin{equation}
T(\eqfs zK)
=
T([z_1,\ldots,z_K, x_1,\ldots,x_N])
=
T(\eqfs xN).
\end{equation}
This sows that \eqref{ext} correctly defines an element
$g\in\Fun{L_X,G}$. It is easily show that $g$ inherits the 
modular property from $T$: putting 
$\bar x=\bigvee_{1\le n\le N}\iota(x_n)$ and
$\bar z=\bigvee_{1\le k\le K}\iota(z_k)$,
\begin{align*}
g(\bar x)
+
g(\bar z)
&=
T(\eqfs xN)+T(\eqfs zK)
\\&=
T([x_1\wedge z_1,\ldots,x_n\wedge z_k,\ldots,x_N\wedge z_K])
+
T([x_1,\ldots,x_N,z_1,\ldots,z_K])
\\&=
g(\bar x\wedge\bar z)
+
g(\bar x\vee\bar z).
\end{align*}
Uniqueness follows once more from the fact that 
$L_X$ is by definition generated by $X$.
\end{proof}

A more difficult question is whether, under the
conditions of Theorem \ref{th ext}, one may 
obtain a modular extension of $f$ to the whole 
of $L$. A positive answer is established in the 
case of a set function exploiting some classical 
results%
\footnote{
There are several results extending a group
valued set functions, among which one should
necessarily include \cite{lipecki_1983}. Nevertheless
these papers consider additional constraints
on the set function, such as positivity or 
countable additivity, and require additional 
structure on $G$, such as order completeness 
or compactness.
}
.

\begin{corollary}
\label{cor ext}
Let $\A$ be a $\cap$ closed family of subsets 
of a non empty set $\Omega$ with $\emp\in\A$
and $\Ring$ a ring containing $\A$. Any 
$F\in\Fun[m]{\A,G}$ with $F(\emp)=0$ admits
a modular extension to $\Ring$. If $\Ring$ 
is the smallest ring containing $\A$ then such 
extension is unique and, in case $G$ is partially
ordered, positive, if $F$ is monotonic.
\end{corollary}

\begin{proof}
Form the lattice $\A_1$ of all finite unions of
sets from $\A$. By Theorem \ref{th ext}, $F$ 
admits a unique modular extension $F_1$ to 
$\A_1$. Pettis then proved \cite[Theorem 1.2]{pettis} 
that a modular set function on a lattice of sets
which takes values in an Abelian group and vanishes 
on the empty set admits a unique, finitely additive 
extension to the generated ring, $\Ring_0$. The
extension from $\Ring_0$ to $\Ring$ follows
along well known arguments that we just adapt
to account for the fact that $F$ takes values in 
$G$ rather than $\R$

We notice first that each function $f\in\Fun{\Omega,G}$
of the form $f=\sum_{n=1}^Na_n\set{R_n}$ with
$R_1,\ldots,R_N\in\Ring_0$ and 
$a_1,\ldots,a_N\in\mathbb Z$ (in symbols 
$f\in\Sim(\Ring_0)$) may be written in the form
\begin{equation}
\label{disjoint}
f
	=
\sum_{\{\emp< b\le[N]:a_b\ne0\}}a_b\set{R_b}
\qtext{with}
R_b
	=
\bigcap_{n\in b}R_n\cap\bigcap_{j\notin b}R_j^c
\in
\Ring_0
\qtext{and}
a_b=\sum_{n\in b}a_n\in G
\end{equation}
and that
$\sum_{n=1}^Na_nF_1(R_n)
	=
\sum_{\emp< b\le[N]}a_bF_1(R_b)
$, by modularity. Thus, we are allowed to interpret
the quantity $\sum_{n=1}^Na_nF_1(R_n)$ as the
image $H(f)$ of $f$ under a function $H:\Sim(\Ring_0)\to G$.
If $A\in\Ring\setminus\Ring_0$, $b,b'\in\mathbb Z$
$f,f'\in\Sim(\Ring_0)$ and $f+b\set A=f'b'\set A$
then, upon writing $f'-f\in\Sim(\Ring_0)$ in the form
\eqref{disjoint},
$(b-b')\set A
	=
\sum_{\emp< b\le[N]}a_b\set{R_b}$. 
This implies $A=\bigcup_b R_b$,
a contradiction, unless $b=b'$ and $f=f'$. We can 
then define the extension 
$H_A:\Sim(\Ring_0\cup\{A\})\to G$ of $H$ by
letting $H_A(f+b\set A)=H(f)+bg_A$ for some
fixed $g_A\in G$. By transfinite induction%
\footnote{
See the details in \cite[3.2.5]{rao}; the original 
argument is due to Tarski.
}, 
we obtain an extension $\bar H$ of $H$ to
$\Sim(\Ring)$. It is then enough to set
$\bar F=\bar H(\set R)$ for all $R\in\Ring$
since 
$\bar F(R_1\cup R_2)
+
\bar F(R_1\cap R_2)
=
\bar H(\set{R_1\cup R_2}+\set{R_1\cup R_2})
=
\bar H(\set{R_1}+\set{R_2})
=
\bar F(R_1)+\bar F(R_2)$.
\end{proof}

Given that for a family of sets union may well serve
as a semilattice operation, Corollary \ref{cor ext}
may equally well be stated for a $\cup$ semilattice
of sets. It is also implicit in Corollary \ref{cor ext}
a finitely additive version of Dynkin's lemma:

\begin{corollary} 
\label{corollary Dynkin} 
Two finitely additive probabilities which agree on a 
semilattice of sets also agree on the generated algebra.
\end{corollary}

\section{The extension of modular functions}
\label{sec extension}

In this section we introduce sets of the form
\begin{equation}
\label{I}
I(x)
	=
\{z\wedge x:z\in X\}
\qquad
x\in X
\end{equation}
and let $\Ring_X$ be the ring generated by these sets.
Theorems \ref{th canonical} and \ref{th embed} imply
the following:

\begin{theorem}
\label{th Stiltijes}
The groups $\Fun{X,G}$ and $\Fun[m]{\Ring_X,G}$ 
correspond isomorphically via the identity
\begin{equation}
\label{Stiltijes}
\mu(I(x))
	=
f(x)
\qquad
x\in X.
\end{equation}
Moreover, if $\mu\in\Fun[m]{\Ring_X,G}$ and 
$f\in\Fun{X,G}$ satisfy \eqref{Stiltijes}, then
(i)
$f\in\Fun[m]{X,G}$ if and only if 
\begin{equation}
\label{Stiltijes modular}
\mu\Big(\bigcup_{x\in a}I(x)\Big)
	=
\mu\big(I(\vee a)\big)
\qquad
a\in\mathcal X,\ 
\vee a\in X,
\end{equation}
(ii) if $G$ is partially ordered, then $\mu$ is positive
if and only if $f$ satisfies \eqref{monotonic}.
\end{theorem}

\begin{proof}
$I$ embeds $X$ into $2^X$. Denote by
$L$ the sublattice of $2^X$ generated by $I[X]$. By 
Theorem \ref{th embed} there exists 
$\iota\in\Hom\Xi L$ such that
\begin{equation}
\iota(\eqfs xN)
	=
\bigcup_{1\le n\le N}I(x_n)
\qquad
\eqfs xN\in\Xi.
\end{equation}
In the present case $\iota$ is a lattice
isomorphism of $\Xi$ and $L$. In fact,
$\bigcup_{1\le n\le N}I(x_n)
	\le
\bigcup_{1\le k\le K}I(z_k)$
implies that for each $1\le n\le N$ there exists
$1\le k\le K$ such that $x_n\in I(z_k)$ i.e.
$x_n\le z_k\le x_n$ so that $\eqfs xN\le\eqfs zK$. 
Thus, letting $T$ denote the canonical representation 
of $f$, we obtain
\begin{equation}
\label{mu}
\mu_0
	=
T\circ\iota^{-1}
\in\Fun[m]{L,G}.
\end{equation}
Then $\mu_0$ satisfies \eqref{Stiltijes} so
that to each $f\in\Fun{X,G}$ corresponds an element 
$\mu_0\in\Fun[m]{L,G}$. Moreover, by modularity
such $f\to\mu_0$ is unique. We further deduce 
from Corollary \ref{cor ext} that associated with $f$ 
there is a unique $\mu\in\Fun[m]{\Ring,G}$, i.e. that 
there exists an injective map $f\to\mu$ of $\Fun{X,G}$
into $\Fun[m]{\Ring,G}$. That this map is onto is obvious.

If $\bigvee_{1\le n\le N}x_n$ exists in $X$,
then
\begin{align*}
\mu\Big(\bigcup_{1\le n\le N}I(x_n)\Big)
=
T(\eqfs xN)
\end{align*}
so that \eqref{Stiltijes modular} is clearly equivalent
to Condition \tiref{i} of Theorem \ref{th canonical},
i.e. to $f$ being modular.
\end{proof}

Of course, if $X$ admits a greatest element, then
$X\in\Ring$ and $\Ring$ is an algebra. We refer
to $\mu$ satisfying \eqref{Stiltijes} as the
distribution of $f$. We stress that this construction
does not impose any restriction on $f$ nor on $G$.
As a drawback, we cannot reach any conclusion
concerning boundedness or countable additivity.

Nevertheless the construction of distributions and
the relative ease in extending set functions to
larger domains suggests the possibility that
a function on a semilattice and values in a general
groups my be extended to a larger domain than
the lattice generated by $X$. Under some conditions
concerning the relation between the original domain
and its extension.

\begin{theorem}
\label{th extension}
Let $X$ be a $\wedge$ closed subset of a lattice 
$L$ and $\iota:X\to L$ the inclusion map. 
Each $f\in\Fun{X,G}$ admits an extension $g\in\Fun{L,G}$. 
Moreover, $g$ is modular if so is $f$ and if 
\begin{equation}
\label{dense}
(\forall a\subset L,\text{ finite})
(\forall y\in a)
(\forall x\in X)
(\exists x_y^a\in X):
x_y^a\le x\wedge y
\qtext{and}
\bigvee_{y\in a}y\wedge\iota(x)
	=
\bigvee_{y\in a}\iota(x^a_y).
\end{equation}
\end{theorem}

\begin{proof}
Let $\mu\in\Fun[m]{\Ring_X,G}$ be the distribution
of $f$, define the map $J\in\Fun{X,2^L}$ by letting
\begin{equation}
\label{J}
J(x)
	=
\{y\in L:y\le x\}
\qquad
x\in X
\end{equation}
and denote by $\mathscr T_0$ the ring generated by
such sets $J(x)$. Clearly, $I(x)=J(x)\cap X$ (see \eqref{I}) 
and $\Ring_X=\mathscr T_0\cap X$. Define 
$\nu_0\in\Fun[m]{\mathscr T_0,G}$
upon letting
\begin{equation}
\nu_0(B)
	=
\mu(B\cap X)
\qquad
B\in\mathscr T_0.
\end{equation}
For each finite sequence $a_1,\ldots,a_N$ of finite 
subsets of $L$ and each pair $x,x'\in X$,
consider the inclusion
\begin{equation}
\label{incl}
J(x)\setminus J(x')
	\subset 
\bigcup_{n=1}^NJ(\vee a_n)\setminus\bigcup_{y\in a_n}J(y).
\end{equation}
We claim that \eqref{incl} implies 
$\nu_0(J(x)\setminus J(x'))
	\le
0$.
The proof is by induction on $N$. If $N=1$ and 
$y\in a_1$, \eqref{incl} implies 
$J(x\wedge y)\setminus J(x')=\emp$, that is 
$\iota(x)\wedge y\le x'$ for all $y\in a_1$.
Thus, 
$x'
\ge
\vee a_1\wedge\iota(x)
=
x$, so the claim is verified. 
Assuming its validity up to $N-1$, fix
$y\in a_N$ and $x^{a_N}_y$ as in \eqref{dense}. 
Then $J(x^{a_N}_y)\setminus J(x')$ is covered
by the union of $N-1$ elements in \eqref{incl}
so that 
$\nu_0(J(x^{a_N}_y)\setminus J(x'))
	=
0$
by the inductions step. Therefore,
\begin{align*}
\nu_0(J(x)\setminus J(x'))
&=
\nu_0(J(x)\setminus J(x')\cup J(x^{a_N}_y))
\\&=
\mu(I(x)\setminus I(x')\cup I(x^{a_N}_y))
\\&=
\mu\Big(I(x)\setminus I(x')
\cup
\bigcup_{1\le n\le N}\bigcup_{y\in a_N} I(x^{a_N}_y)\Big)
\\&=
\mu\big(I(x)\setminus I(x')\cup I(x)\big)
\\&=
0.
\end{align*}
This follows from  the assumption that
$\bigvee_{\substack{1\le n\le N\\y\in a_N}} \iota(x^{a_N}_y)
=
x$
and that $\mu$ satisfies \eqref{Stiltijes modular}.

Thus \cite[Proposition 1]{lipecki_1984} applies and
delivers the existence of an extension $\nu_1$ of 
$\nu_0$ from $\mathscr T_0$ to $\mathscr T_1$, 
the ring generated by $\mathscr T_0$ and the family 
$J(y)$ with $y\in L$, that satisfies the additional 
property  
\begin{equation}
\nu\Big(J(\vee a)\setminus\bigcup_{y\in a}J(y)\Big)
=
0
\qquad
a\subset L,\ a\text{ finite}.
\end{equation}
It is then enough to let 
$
g(y)=\nu(J(y))
$ 
for every $y\in L$ and modularity follows from 
\eqref{Stiltijes modular}.
\end{proof}

A clear case in which condition \eqref{dense} holds is
when $L$ is generated by $X$, the case dealt with in
Theorem \ref{th ext}. The following is another
quite general situation.

\begin{corollary}
\label{cor extension}
Let $X$ be a subset of a lattice $L$ with the property
that $x\wedge y\in X$ for all $x\in X$ and $y\in L$.
Then, every $f\in\Fun[m]{X,G}$ admits an extension 
$g\in\Fun[m]{L,G}$. 
\end{corollary}

Corollary \ref{cor extension} applies, for example,
to the case in which $X$ consists of the subsets
of a topological space which have compact closure.
Another easy example is constructed when $\Omega$ 
is some arbitrary topological space, the space 
$\Fun\Omega$ is endowed with pointwise minimum
and $X$ consists of functions with compact support.
It is equally easy to see that in both these examples
property \eqref{groemer} applies so that the definition
of modularity takes a conveniently simple form.

\bibliographystyle{acm}
\bibliography{/MathBib}

\end{document}